\documentclass[]{article}

\usepackage{amssymb,amsthm, amsmath}
	\theoremstyle{definition}
		\newtheorem{thm}{Theorem}
		\newtheorem{lemma}[thm]{Lemma}
		\newtheorem{defin} [thm] {Definition}
		\newtheorem{prop}[thm]{Proposition}
		\newtheorem{ques}[thm]{Question}
        \newtheorem{cor}[thm]{Corollary}
	\theoremstyle{remark}
		\newtheorem*{rmk}{Remark}

\title{An Elementary Approach on Left-Orderability, Cables of Torus Knots and Dehn Surgery}
\author{Jianhui Li and Chun-Yin Siu}

\begin{document}

\renewcommand{\thefootnote}{\fnsymbol{footnote}}
\footnotetext{
2010 Mathematics Subject Classification: 57M25 (primary), 57M50 and 57M99}
\footnotetext{keywords: Dehn surgery, left order, cable knot, L-space}
\renewcommand{\thefootnote}{\arabic{footnote}}

\maketitle

\begin{abstract}
Motivated by Clay and Watson\text{'}s question \cite{CW decay} on left-orderability of the fundamental group of the resultant space of an $r'$-surgery on the $(p, q)$-cable knots for $r' \in (pq-p-q,pq)$, this paper proves by elementary means that for specific pairs of $(p,q)$-cable knots of torus knots, $r' \in [pq-1,pq]$ gives a surgery yielding non-left orderable fundamental groups.

\end{abstract}

\section{Introduction} \label {Introduction}

A non-trivial group is said to be left-orderable iff there is a strict total order on the group that is preserved by left multiplication. Motivated by Clay and Watson\text{'}s \cite{CW decay} result on decayed knots, this paper studies left orderability of the fundamental groups of a special class of spaces obtained from performing Dehn surgery on cable knots. In Clay and Watson\text{'}s \cite{CW decay} paper, the following question was raised.

\begin{ques} \label{CW ques}
Let $r_0$ be a positive rational number and $K$ be a nontrivial knot in $S^3$ with canonical basis $ \{\mu, \lambda \}$. Denote $\mu^m \lambda^n$ by $\alpha_{m/n}$ whenever $m$ and $n$ are relatively prime. Suppose for every left ordering on $\pi_1(K)$, either $\alpha_r$ is positive for every $r > r_0$, or $\alpha_r$ is nonpositive for every $r > r_0$. If $C$ is a $(p, q)$-cable of $K$ with $q/p > r$, is $\pi_1(S^3_r(C))$  left orderable for $pq - p - q < r < pq$?
\end{ques}

Without further information of the knot concerned, it is in general very difficult to address Question \ref{CW ques} by elementary means. Therefore, we have specialized to torus knots \footnote{It was shown in \cite{CW decay} that, for the $(x,y)$-torus knot, $r_0 = {xy - 1}$ satisfies the assumption of Question \ref{CW ques}.}, whose fundamental groups are well understood and well behaved, and answered the question for $r$ near $pq$. More specifically, the following theorem is proven.

\begin{thm} \label{grand total}
Let $K$ be the $x,y$-torus knot. Let $p\geq2$ be an integer and $q=pxy-1$. If $r' \in [pq-1,pq]$. Then the fundamental group of $S^3_{r'}(C_{p,q}(K))$ is not left-orderable.
\end{thm}

The choice of the assumptions will be explained as the proof unfolds.

This result is in fact special cases of a general phenomenon, as $S^3_{r'}(C_{p,q}(K))$ is a graph manifold L-space, which was proven to be not left-orderable in \cite{BC} and \cite{HRRW}. This general fact, however, is proven indirectly by means of foliation. Our interest here, on the contrary, is in proving the result by elementary means.

The paper is organized as follows. Section \ref{Background} contains the background materials on left orderable groups, torus knots and cable knot, and Dehn surgery. The general case is also reviewed in this section.
Section \ref{The main theorem} presents the proof of Theorem \ref{grand total}, which consists of two parts. The first part proves Theorem \ref{grand total} for the special case of an infinite sequence in $[pq-1,pq]$ completely by elementary manipulation of the group relations, 
and the second part completes the proof with an argument similar to Clay and Watson's in \cite{CW decay} and \cite{CW obstruct}.

The authors would like to thank their supervisor Zhongtao Wu for his continued guidance and support.

\section{Background} \label{Background}

\subsection{Left order}

\begin{defin}
A left order on a group $G$ is a strict total ordering $<$ of elements of $G$ such that $ga < gb$ whenever $a<b$. The positive cone of the left order $<$ is $P=\{g \in G|g>1\}$. Two elements are said to have the same sign iff either they are both in $P$, they are both in $P^{-1}=\{g^{-1} \in G|g>1\} = \{ h \in G| h^{-1} > 1\}$, or they are both $1$. A nontrivial group is said to be left-orderable if it has a left order. 
\end{defin}

\begin{prop} \label{LO}.
	\begin{enumerate}
		\item $g < h$ iff $g^{-1} h \in P$
		\item $P$ is closed under multiplication.
		\item \label{part} $G = P \sqcup \{1 \} \sqcup P^{-1}$.
		\item If $g$ and $h$ commute, then for every positive integer $n$, $g < h$ iff $g^n<h^n$. In particular, for every positive integer $n$, $g > 1$ iff $g^n > 1$.
	\end{enumerate}
\end{prop}
\begin{proof}
The first claim follows by multiplying $g$ or $g^{-1}$ on the left. For the second claim, $gh>g \cdot 1 = g >1$ whenever $g, h \in P$. For the third claim, it is clear that $P$ and $\{ 1 \}$ are disjoint, and hence so are $P^{-1}$ and $\{ 1 \}$; that $P$ and $P^{-1}$ are disjoint follows from the second claim and that $P$ and $\{ 1 \}$ are disjoint. If $g < 1$, then $g^{-1} = g^{-1} \cdot 1 > g^{-1} g = 1$, hence $g \in P^{-1}$; this proves the third claim. Finally,  for the last claim, if $g$ and $h$ commute and $g < h$, then $(g^{-1} h)^n = g^{-n} h^n$, hence by the first claim, it suffices to prove the particular case. The 'only if' part follows from the closure of $P$. For the 'if' part, if $g \le 1$, then $g^{-1} \ge 1$, and hence by closure of $P$, $g^{-n} \ge 1$, and the result then follows.
\end{proof}

\subsection{Torus knot and Cable knot}

The following two propositions are proven in \cite{CW obstruct} and \cite{CW decay} respectively.

\begin{prop} \label{torus}
The fundamental group of a torus knot $T_{x,y}$ is $\langle a,b|a^x=b^y \rangle$ with canonical basis $\{\mu,\lambda\}$, where $\mu = b^ja^i$, $\lambda = \mu ^{-xy}a^x$ and $xj+yi=1$.
\end{prop}
\begin{prop} \label{cable}
Let $k$ be a knot with canonical basis $\{\mu,\lambda\}$. The fundamental group of the $p,q$-cable knot $C_{p,q}(K)$ on $K$ is $\pi_1(K)\ast_{\mu^q\lambda^p=t^p}\mathbb{Z}$, with canonical basis $\{\mu_C,\lambda_C\}$, where $\mu_C = \mu^u\lambda^vt^{-v}$ and $\lambda_C = \mu_C ^{-pq}t^p$, and $pu-qv=1$.
\end{prop}

\begin{rmk}
Note that in the two propositions above, $i, j, u$ and $v$ are not uniquely determined, yet different choices give the same meridians and longitudes. Indeed, for the case of a torus knot, since $a^x=b^y$ , $\mu = b^ja^i=b^{j-ny}a^{i+nx}$; for the case of a cable knot, since $\mu^q\lambda^p=t^p$ and the peripheral elements $\mu,\lambda$ commute, $\mu_C = \mu^u\lambda^vt^{-v} = \mu^u\lambda^vt^{np}t^{-v-np}=\mu^{u+nq}\lambda^{v+np}t^{-(v+np)}$.
\end{rmk}

One challenge of studying fundamental group is that the group is typically nonabelian. However, in the case of torus knots and cable knots, many elements commute with each other, as summarised by the following proposition.

\begin{prop} \label{commute}
Using the notation above, the following pairs and triples commute in every quotient of the above groups .
	\begin{enumerate}
		\item $a^x$ and $b$
		\item $a$ and $b^y$
		\item $\mu$, $\lambda$ and $t^p$
		\item $\mu_C$ and $\lambda_C$ and $t^p$
	\end{enumerate}
\end{prop}
\begin{proof}
The first two claims follows directly from that relation $a^x = b^y$.For the third claim, that $\mu$ and $\lambda$ commute follows from the fact that they are peripheral elements; recalling the group relation $\mu^q\lambda^p=t^p$, since $t^p$ is a product of $\mu$ and $\lambda$, which commute, $t^p$ commutes with them. The proof of the last claim is similar.
\end{proof}

\subsection{Dehn surgery}
The following proposition is well known.

\begin{prop}
The fundamental group of the resultant space $S^3_{m/n}(K)$ of an $\frac{m}{n}$-surgery on a knot $K$ is $\pi_1(K)/\langle \mu^m \lambda^n \rangle$, where $\{\mu,\lambda\}$ is the canonical basis.
\end{prop}

\subsection{The general phenomenon}  
As mentioned in the introduction, Theorem \ref{grand total} is a special case of a general phenomenon.

Let $C$ be the $(p,q)$-cable of the $(x,y)$-torus knot. Then it has genus $g(C) = \frac{1}{2} [(p-1)(q-1) + p(x-1)(y-1)]$.   

In \cite[Theorem 1.1]{BC} and \cite[Theorem 2]{HRRW}, it was proven that a graph manifold L-space has a non-left-orderable fundamental group, so it suffices to consider the surgeries that give rise to graph manifold L-spaces.

It was first established in \cite[Theorem 7.5]{G} that a surgery on $C$ always gives a graph manifold.
By \cite[Theorem 1.10]{H}, $C$ is an L-space knot,
hence $S_r^3(C)$ is an L-space iff $r \geq 2g(C) - 1$.  
Therefore, $S_r^3(C)$ has a non-left-orderable fundamental group iff $r \geq 2g(C) - 1 = (p-1)(q-1) + p(x-1)(y-1) - 1$. Plugging in $q = pxy - 1$ and rearranging shows $(p-1)(q-1) + p(x-1)(y-1) - 1 = (pq - 1) - [p(x + y) - 2]$. Since $p(x + y) - 2 > 0$, this implies Theorem \ref{grand total}.

\section{Proof of Theorem \ref{grand total}} \label {The main theorem}

Combining propositions from the previous section, the knot group $G$ of $C_{p,q}(T_{x,y})$ is generated by $a$, $b$ and $t$, with the following relations.

\begin{equation} \label{ab}
	a^x=b^y \end{equation}
\begin{equation} \label{mu}
	\mu = b^ja^i \end{equation}
\begin{equation} \label{lambda}
	\lambda = \mu ^{-xy}a^x \end{equation}
\begin{equation} \label{t}
	\mu^q\lambda^p=t^p \end{equation}

where
\begin{equation} \label{xy}
	xj+yi=1 \end{equation}
\begin{equation} \label{pq}
	pu-qv=1 \end{equation}

The meridian $\mu_C$ and the longitude $\lambda_C$ are
\begin{equation} \label{mu_C}
	\mu_C = \mu^u \lambda^v t^{-v} \end{equation}
\begin{equation} \label{lambda_C}
	\lambda_C = \mu_C ^{-pq}t^p \end{equation}

The fundamental group $H$ of $S^3_{m/n}(C_{p,q}(T_{x,y}))$ is then the quotient of $G$ by the relation \begin{equation} \label{surgery} \mu_C^m\lambda_C^n=1 \end{equation}.

The following lemma will be used later.
\begin{lemma} \label{same sign}
Suppose $q = pxy - 1$. Then in every quotient $H$ of $G$, $t^p = a^{xp - i} b^{-j}$. Hence if $H$ is left-orderable, $t$, $a$ and $b$ have the same sign.
\end{lemma}
\begin{proof}
By the remark after Propositions \ref{torus} and \ref{cable}, $i$ can be chosen such that $0 < i < x$, and hence $j < 0$; since $q = pxy - 1$, $u$ and $v$ can be chosen to be $xy$ and $1$ respectively. Then
\begin{align*}
t^p
  &= \mu^q \lambda^p                &&\text{by equation (\ref{t})}
\\&= \mu^{-1} (\mu^{xy} \lambda)^p	&&\text{by the choice of $q$ and commutivity of $\mu$ and $\lambda$}
\\&= \mu^{-1} a^{xp}                &&\text{by equation (\ref{lambda})}
\\&= a^{-i} b^{-j} a^{xp}           &&{\text{by equation (\ref{mu})}}
\\&= a^{xp - i} b^{-j}				&&{\text{by commutivity of $a^{x}$ and $b$}}
\end{align*}
Equation (\ref{ab}) implies $a$ and $b$ have the same sign. Since $p \ge 1$, $xp - i > 0$. Therefore, $t^p = a^{xp - i} b^{-j}$, and hence $t$, has the same sign as $a$ and $b$.
\end{proof}

We begin the proof by observing the fundamental group of near-$pq$-surgeries of cable knots, or more specifically, $pq - 1/\beta$-surgeries. It will be proven in Theorem \ref{discrete} that, for every $\beta \ge 1$, if $x$, $y$, $p$ and $q$ are nice, or more precisely, as assumed in Theorem \ref{grand total}, the generators $a$, $b$ and $t$ have contradictory signs under any left order, and hence $H$ cannot be left-orderable.

\begin{lemma}
If $H$ is left-orderable and $m/n = pq - 1/\beta$, then \begin{equation} \label{r'=pq-1/beta} t^{p\beta + v} = \mu^u \lambda^v \end{equation}.
\end{lemma}
\begin{proof}
Let $m=pq\beta-1$ and $n=\beta$. Then $m$ and $n$ are relatively prime and $m/n = pq - 1/\beta$. Equation (\ref{surgery}) then implies $\mu_C^{pq\beta-1}\lambda_C^{\beta}=1$. Since $\mu_C$ and $\lambda_C$ commute, rearrangement gives $(\mu_C^{pq} \lambda_C)^{\beta}=\mu_C$. The lemma follows from equations (\ref{mu_C}) and (\ref{lambda_C}).
\end{proof}



Note the similarity between equations (\ref{r'=pq-1/beta}), (\ref{lambda}) and (\ref{t}), each of which expresses a power of a generating element as the product of powers of $\mu$ and $\lambda$. The assumption of Theorem \ref{grand total} can now be used to obtain relationships among the generating elements $a$, $b$ and $t$ from these equations, and hence the conclusion of Theorem \ref{grand total} can be established for the sequence $(r'_\beta) = (pq - 1/\beta)$ converging to $pq$, as in the following theorem.

\begin{thm} \label{discrete}
Under the assumptions of Theorem \ref{grand total}, for every positive integer $\beta$, the fundamental group of $S^3_{r'_\beta}(C_{p,q}(K))$, where $r'_\beta = pq-1/\beta$, is not left-orderable.
\end{thm}

\begin{proof}
Suppose the group is left-orderable. As in the proof of Lemma \ref{same sign}, assume $u = xy$, $v = 1$, $0 < i < x$ and $j < 0$.

Equation (\ref{r'=pq-1/beta}) then becomes $t^{p \beta + 1} = \mu^{xy} \lambda$, which combines with equations (\ref{lambda}) and (\ref{ab}) to give \begin{equation} \label{eq1} t^{p \beta + 1} = a^x = b^y. \end{equation}
Lemma \ref{same sign} shows $t^p = a^{px - i} b^{-j}$. Dividing this by equation (\ref{eq1}) gives \begin{equation} t^{-[p(\beta - 1) + 1]} = a^{(p-1)x - i} b^{-j}. \end{equation}
Then since $p \ge 2 $, $(p-1)x - i$ and $-j$ are both positive, and hence $a^{(p-1)x - i} b^{-j}$ has the same sign as $a$, $b$ and $t$. However, $-[p(\beta - 1) + 1]$ is negative for positive $\beta$, and hence $t^{-[p(\beta - 1) + 1]}$ does not have the same sign as $t$. This contradicts with Lemma \ref{same sign}.
\end{proof}

Observe that Theorem \ref{discrete} can be interpreted as follows.

\begin{cor} \label{cor of discrete case}
Under the same assumptions of Theorem \ref{grand total}, in every left-orderable quotient of $G$, $\mu_C^{pq \beta-1} \lambda_C^\beta \neq 1$.
\end{cor}
\begin{proof}
The theorem shows that equation (\ref{surgery}) obstructs left-orderability of a quotient of $G$.
\end{proof}

Clay and Watson's obstruction \cite{CW obstruct} can be used to fill in the gap between points in the sequence in the above theorem. However, its full generality is not needed in this case. The elementary technique in \cite{CW decay} and \cite{CW obstruct} of applying Cramer's rule suffices.

\begin{prop} \label{sim CW thm}
Let $\{M,L\}$ be the canonical basis for the peripheral subgroup of a knot group $G$. If, in a left order of a quotient of $G$, $M^{m_0} L^{n_0}$ and $M^{m_1} L^{n_1}$ have the same sign, then they have the same sign as $M^m L^n$ whenever $\frac{m}{n} \in (\frac{m_0}{n_0},\frac{m_1}{n_1})$.
\end{prop}
\begin{proof}
Let
$\Delta_0 = \left | \begin{array}{cc}
n & n_1 \\
m & m_1  \end{array} \right |$,
$\Delta_1 = \left | \begin{array}{cc}
n_0 & n \\
m_0 & m \end{array} \right |$,
and $\Delta = \left | \begin{array}{cc}
n_0 & n_1 \\
m_0 & m_1 \end{array} \right |$.
By Cramer's rule,
\begin{equation}
	\begin{split}
		n_0 \Delta_0 + n_1 \Delta_1 =& n \Delta \\
		m_0 \Delta_0 + m_1 \Delta_1 =& m \Delta
	\end{split}
\end{equation}

This implies $(M^{m_0} L^{n_0})^{\Delta_0} (M^{m_1} L^{n_1})^{\Delta_1} = (M^m L^n)^\Delta$.
By assumption, $M^{m_0} L^{n_0}$ and $M^{m_1} L^{n_1}$ have the same sign. Since $\frac{m_0}{n_0} < \frac{m}{n} < \frac{m_1}{n_1}$, $\Delta_0$, $\Delta_0$ and $\Delta$ are all positive. Then $(M^m L^n)^\Delta$, and hence $M^m L^n$, has the same sign as $M^{m_0} L^{n_0}$ and $M^{m_1} L^{n_1}$.
\end{proof}

Theorem \ref{grand total} can now be proven.

\begin{proof} [proof of Theorem \ref{grand total}]

Equation (\ref{lambda_C}) implies $t^p = \mu_C^{pq}\lambda_C$. Then by Proposition \ref{sim CW thm}, it suffices to show $\mu_C^{pq-1}\lambda_C$ has the same sign as $t$ in every left order. Indeed, if the fundamental group of $S^3_{r'}(C)$, which is $G/\langle \mu_C^{m'} \lambda_C^{n'} \rangle$ with $r' = m'/n'$, is left-orderable, Proposition \ref{sim CW thm} implies that for $r' = m'/n' \in (pq-1,pq)$, $\mu_C^{pq-1}\lambda_C$ has the same sign as $\mu_C^{m'} \lambda_C^{n'} = 1$, and hence $\mu_C^{pq-1}\lambda_C$ = 1. This contradicts with Corollary \ref{cor of discrete case}.

As in the proof of Lemma \ref{same sign}, assume $u = xy$, $v = 1$, $0 < i < x$ and $j < 0$.

\begin{align*}
\mu_C^{pq-1} \lambda_C
  &= \mu_C^{-1} (\mu_C^{pq} \lambda_C)
\\&= \mu_C^{-1} t^p                            &&\text{by equation (\ref{lambda_C})}
\\&= t (\lambda^{-1} \mu^{-xy}) t^p            &&\text{by equation (\ref{mu_C}) and the choice of $u$ and $v$}
\\&= t a^{-x} t^p                              &&\text{by equation (\ref{lambda})}
\\&= t a^{-x} (a^{xp-i} b^{-j})                &&\text{by Lemma \ref{same sign}}
\\&= t a^{x(p-1)-i} b^{-j}
\end{align*}

By Lemma \ref{same sign}, $t$, $a$ and $b$ have the same sign. Since $p \geq 2$, the choice of $i$ and $j$ implies,  $x(p-1) - i$ and $-j$ are positive, so $\mu_C^{xy-1} \lambda_C = t a^{x(p-1)-i} b^{-j}$ has the same sign as $t$. The contradiction then follows.

\end{proof}


\begin{rmk}
	Note that for $r'=pq-1/\beta$, Theorem \ref{discrete} and the above proof gives two proofs of the non-left orderability of the corresponding fundamental group. Even though the proof of Theorem \ref{discrete} is longer, it is also more elementary.
\end{rmk}

\end{document}